\documentclass[reqno,12pt]{amsart}

\usepackage{amsmath, amsthm, amsfonts, amssymb}
\usepackage{amsmath}
\usepackage{amsfonts}
\usepackage{amssymb}
\usepackage{amsthm}
\usepackage{amstext}

  \theoremstyle{plain}%default
  \newtheorem{theorem}{Theorem}[section]
  \newtheorem{proposition}[theorem]{Proposition}
  \newtheorem{lemma}[theorem]{Lemma}
  \newtheorem{corollary}[theorem]{Corollary}
  \theoremstyle{definition}

  \newtheorem{example}[theorem]{Example}
  \theoremstyle{remark}

 \numberwithin{equation}{section}

\author{V. Manuilov}

\date{}

\address{Moscow Center for Fundamental and Applied Mathematics {\rm and} Moscow State University,
Leninskie Gory 1, Moscow, 
119991, Russia}

\email{manuilov@mech.math.msu.su}

\thanks{The research was supported by RSF, project No. 21-11-00080}

\title{Hilbert $C^*$-modules related to discrete metric spaces}

\begin{document}

\maketitle

\begin{abstract}

It is shown that the metric on the union of the sets $X$ and $Y$ defines a Hilbert $C^*$-module over the uniform Roe algebra of the space $X$ with a fixed metric $d_X$. A number of examples of such Hilbert $C^*$-modules are described.

\end{abstract}

\section*{Introduction} 

Hilbert $C^*$-modules (\cite{Paschke}) are a natural generalization of Hilbert spaces, in which the ``scalar'' product takes  values in some $C^*$-algebra instead of the field of complex numbers. Although many properties of Hilbert $C^*$-modules are similar to those of Hilbert spaces, there are several important differences, among which are the following: not every closed submodule is orthogonally complemented, and not every functional is defined as a ``scalar'' product by some element (Riesz theorem). 
If $M$ is a (right) Hilbert $C^*$-module over the $C^*$-algebra $A$, then it is natural to call a bounded $A$-linear map of $M$ into $A$ an $A$-linear functional. The set of all such mappings constitutes the dual module $M'$, on which there is the structure of a right $A$-module, but, in general, there is no $A$-valued ``scalar'' product. Moreover, the bidual module $M'' $, dual to the module $M'$ is a Hilbert $C^*$-module, and there is an isometric embedding $M\subset M''\subset M'$ (\cite{Frank1}, \cite{Paschke2}).  

The standard Hilbert $C^*$-module $l_2(A)$ is a (right) $A$-module of sequences $(a_n)_{n\in\mathbb N}$, where $a_n\in A$, $ n\in\mathbb N$, and $\sum_{n=1}^\infty a_n^*a_n$ converges in $A$ (in the norm). In this case, the dual module $l_2(A)'$ consists of sequences $(a_n)_{n\in\mathbb N}$, for which the partial sums $\sum_{n=1}^m a_n^*a_n$, $m\in\mathbb N$, are uniformly bounded, but for the bidual module $l_2(A)''$ there is generally no good description (but the description of $l_2(A)''$ is known in the case when $A$ is commutative \cite{FMT}).

Another important difference between Hilbert $C^*$-modules and Hilbert spaces is that a finitely generated Hilbert $C^*$-module does not have to be free, and a countably generated Hilbert $C^*$-module does not have to be standard, for example, $C_0(0,1)$ is not isomorphic to $l_2(C[0,1])$ as modules over $C[0,1]$.
The purpose of this paper is to demonstrate the variety of Hilbert $C^*$-modules using the example of modules over uniform Roe algebras. Information about Hilbert $C^*$-modules can be found in \cite{MT}, and about Roe algebras and underlying metric spaces in \cite{Novak-Yu}, \cite{Roe}.

\section{Hilbert $C^*$-modules over uniform Roe algebras}

We denote the algebra of bounded (respectively, compact) operators of the Hilbert space $H$ by $\mathbb B(H)$ (respectively, $\mathbb K(H)$).

Let $X=(X,d_X)$ be a discrete countable metric space $X$ with the metric $d_X$, $H_X=l^2(X)$ be the Hilbert space of square summable complex-valued functions on $X$ with a standard orthonormal basis consisting of delta functions of points, $\delta_x$, $x\in X$. A bounded operator $T$ on $H_X$ with the matrix $(T_{x,y})_{x,y\in X}$ with respect to the standard basis, i.\,e. $T_{x,y}=(\delta_x,T\delta_y)$, has a \textit{propagation} not exceeding $L$ if $d_X(x,y)\geq L$ implies that $T_{x,y}=0$. The $*$-algebra of all bounded operators of finite propagation is denoted by $\mathbb C_u[X]$, and its norm completion in $\mathbb B(H_X)$ is called the uniform Roe algebra $C^*_u(X)$.

For the set $Y$, let $d$ be the metric on $X\sqcup Y$ coinciding on $X$ with $d_X$, i.\,e. $d|_X=d_X$.

We denote by $\mathbb M_{Y,d}$ the set of all bounded operators of finite propagation $T:H_X\to H_Y$, and by $M_{Y,d}$ its norm closure in the set $\mathbb B(H_X,H_Y)$ of all bounded operators from $H_X$ to $H_Y$.

If the operators $T\in\mathbb B(H_X,H_Y)$ and $R\in\mathbb B(H_X)$ have a finite propagation with respect to the metrics $d$ and $d_X$, respectively, then their composition $TR$ obviously also has a finite propagation with respect to the metric $d$. It follows from the continuity of the composition that the action of the $C^*$-algebra $C_u^*(X)$ on $M_{Y,d}$ is well defined and provides the structure of a right $C_u^*(X)$-module.
Similarly, if $T,S\in\mathbb B(H_X,H_Y)$ are operators of finite propagation with respect to $d$, then $S^*T$ has a finite propagation on $l^2(X)$ with respect to $d_X$, so one can define $\langle S,T\rangle=S^*T\in C_u^*(X)$ and extend it by continuity to a $C_u^*(X)$-valued inner product on the module $M_{Y,d}$.

\begin{lemma}
The module $M_{Y,d}$ is a Hilbert $C^*$-module over $C^*_u(X)$. 

\end{lemma}
\begin{proof}
Evidently, $\|T\|^2=\|\langle T,T\rangle\|$. The remaining properties of Hilbert $C^*$-modules follow from associativity of operator multiplication.
\end{proof}

\begin{example}
Let $Y$ be a one-point space, $Y=\{y_0\}$. Any operator $T:H_X\to H_Y=\mathbb C$ is a functional on $H_X$ and can be approximated by a functional with a finite number of nonzero coordinates; therefore, it is the limit of operators of finite propagation, i.\,e., $M_{y_0}$ can be identified with functionals on $H_X$ and, by the Riesz theorem, with $H_X$.

\end{example}

In \cite{Frank-monotone} it was shown that the structure of a Hilbert $C^*$-module $M$ extends to the dual module $M'$ (making the latter a Hilbert $C^*$-module) if and only if the $C^*$-algebra $A$ is monotone complete. Recall that monotone completeness of $A$ means that any bounded increasing set $\{a_\alpha:\alpha\in I\}$ of self-adjoint elements of the $C^*$-algebra $A$ has the least upper bound $a=\sup\{a_\alpha:\alpha\in I\}$ in $A$.

\begin{theorem}
A metric space $X$ is bounded if and only if the $C^*$-algebra $C_u^*(X)$ is monotone complete.

\end{theorem}
\begin{proof}
If the space $X$ is bounded, then any bounded operator $T:H_X\to H_X$ has a finite propagation, then $C_u^*(X)=\mathbb B(H_X)$ is a von Neumann algebra, hence is monotone complete.

Conversely, suppose that the space $X$ is unbounded and that the algebra $C_u^*(X)$ is monotone complete. Find a sequence of pairs of different points $\{(x_n,y_n)\}_{n\in\mathbb N}$ in $X$ satisfying the condition $d_X(x_n,y_n)>n$. Suppose that the pairs of points $(x_1,y_1),\ldots,(x_n,y_n)$ with the condition $d_X(x_i,y_i)>i$, $i=1,\ldots,n$, have already been found. If the estimate $d_X(x,y)\leq n+1$ were fulfilled for any $x,y\neq x_1,\ldots,x_n,y_1,\ldots,y_n$, then the diameter of $X$ would be finite. Hence, it is possible to find points $x_{n+1},y_{n+1}\in X$ that do not coincide with any of the previous ones and satisfy the condition $d_X(x_{n+1},y_{n+1})\geq n+1$. We find such pairs of points inductively for each $n\in\mathbb N$.

Set 
$$
T^{(n)}_{x,y}=\left\lbrace\begin{array}{cl}1,& \mbox{if\ }(x,y)\in\{(x_i,y_i),(y_i,x_i):i=1,\ldots,n\};
\\0&\mbox{otherwise.}\end{array}\right.
$$
Then the matrix $(T^{(n)}_{x,y})$ defines a bounded self-adjoint operator $T^{(n)}$ of finite rank, hence, a finite propagation, for each $n\in\mathbb N$. Let $T\in C_u^*(X)$ be the least upper bound for the set $\{T^{(n)}\}_{n\in\mathbb N}$. Then $T_{x_n,y_n}\geq 1$ for any $n\in\mathbb N$, which contradicts the fact that $T\in C_u^*(X)$.
\end{proof}

\begin{corollary}
A metric space $X$ is bounded if and only if the structure of a Hilbert $C^*$-module extends from any Hilbert $C^*$-module $M$ over $C_u^*(X)$ to its dual module $M'$.
\end{corollary}

Recall that two metrics, $d_1,d_2$ on a space $Z$ are \textit{coarsely equivalent} \cite{Novak-Yu} if there exists a monotonely increasing function $\varphi$ on $[0,\infty)$ such that $\lim_{t\to\infty}\varphi(t)=\infty$ and one has $d_1(z_1,z_2)\leq\varphi(d_2(z_1,z_2))$ and $d_2(z_1,z_2)\leq \varphi(d_1(z_1,z_2))$ for any $z_1,z_2\in Z$.

\begin{proposition}
Let $d_1$, $d_2$ are metrics on $X\sqcup Y$ with the same restriction on $Y$. They are coarsely equivalent if and only if $M_{Y,d_1}=M_{Y,d_2}$.

\end{proposition}
\begin{proof}
If the metrics are roughly equivalent, then having a finite propagation with respect to one of them is equivalent to having a finite propagation with respect to the other.

Conversely, suppose the metrics are coarsely nonequivalent. Then there is a sequence of pairs of points $(x_n,y_n)$, $x_n\in X$, $y_n\in Y$, $n\in\mathbb N$, such that for one metric the values $d_1(x_n,y_n)$ are uniformly bounded by some constant $C>0$, while the other metric satisfies the estimate $d_2(x_n,y_n)\geq n$. We claim that each point $x_k$ can occur in the sequence $\{x_n\}_{n\in\mathbb N}$ only a finite number of times. Indeed, if $x_k=x_{n_1}=x_{n_2}=\cdots$ then 
$$
d_1(y_{n_i},y_{n_1})\leq d_1(x_k,y_{n_i})+d_1(x_k,y_{n_1})\leq 2C
$$ 
for any $i\in\mathbb N$, while 
$$
d_2(y_{n_i},y_{n_1})\geq d_2(x_k,y_{n_i})-d_2(x_k,y_{n_1})\geq n_i-n_1, 
$$
i.\,e. is not bounded, but the metrics $d_1$, $d_2$ are equal on $Y$, and this contradiction shows that the point $x_k$ can be repeated in the sequence $\{x_n\}_{n\in\mathbb N}$ only a finite number of times. Passing to a subsequence, we may assume that the sequence $\{x_n\}_{n\in\mathbb N}$ does not contain repeating points at all. The same may be assumed for the sequence $\{y_n\}_{n\in\mathbb N}$.

Set
$$
T_{x,y}=\left\lbrace\begin{array}{cl}1,& \mbox{if\ }(x,y)\in\{(x_n,y_n):n\in\mathbb N\};\\
0&\mbox{otherwise.}\end{array}\right.
$$ 
Then the matrix $(T_{x,y})$ defines a bounded operator $T$ from $H_X$ to $H_Y$. It has a finite propagation with respect to the metric $ d_1 $, i.\,e. $T\in M_{Y,d_1}$. If $M_{Y,d_1}=M_{Y,d_2}$ then the operator $T$ should be the limit of finite propagation operators with respect to the metric $d_2$, but this is not the case.   
\end{proof}

The following statements are obvious.
\begin{proposition}
If $d_1(x,y)\leq d_2(x,y)$ for any $x,y\in X\sqcup Y$ then $M_{Y,d_2}\subset M_{Y,d_1}$. 

\end{proposition}
\begin{proposition}
If $Y=Y_1\sqcup Y_2$ then $M_{Y}=M_{Y_1}\oplus M_{Y_2}$.

\end{proposition}

\section{The case $Y=X$ }

Let $Y=X$. To avoid ambiguity, we will denote the first copy of $X$ by $X_0$ and the second copy by $X_1$. Accordingly, the point $x\in X$ will be denoted by $x_0\in X_0$ if it lies in the first copy of $X$, and $x_1\in X_1$ if it lies in the second copy. We will also identify $\mathbb B(H_{X_0},H_{X_1})$ with the algebra $\mathbb B(H_X)$.

The set $S(X)$ of coarse equivalence classes of metrics on $X_0\sqcup X_1$ has the natural structure of an inverse semigroup \cite{M}, where the composition of metrics is given by the formula 
$$
d_1d_2(x_0,z_1)=\inf\nolimits_{y\in X}[d_2(x_0,y_1)+d_1(y_0,z_1)], 
$$
the adjoint (pseudoinverse) metric is given by the formula $d^*(x_0,y_1)=d(y_0,x_1)$, the unit element is given by the metric $d(x_0,y_1)=d_X(x,y)+1$ and the zero element is given by the metric $d(x_0,y_1)=d_X(x,u)+d_X(y,u)+1$ with a fixed point $u\in X$ (recall that for a metric on $X_0\sqcup X_1$ it suffices to define distances between points lying in different copies of the space $X$).

It is clear that if the metrics $d_1$ and $d_2$ are coarsely equivalent then $M_{X,d_1}=M_{X,d_2}$. Thus, we have the Hilbert $C^*$-module $M_{X,d}$ for each coarse equivalence class $s=[d]\in S(X)$. The collection of these Hilbert $C^*$-modules forms a \textit{Fell bundle} in the sense of Definition 2.1 from \cite{exel} (the mapping $M_{X,d_1}\otimes M_{X,d_2}\to M_{X,d_1d_2}$ is given by composition, see \cite{M-RJMP}).

\begin{example}
Let $A\subset X$. Define the metric $d^A$ on $X_0\sqcup X_1$ by 
$$
d^A(x_i,y_i)=d_X(x,y), \quad i=0,1; 
$$
$$
d^A(x_0,y_1)=\inf\nolimits_{z\in A}[d_X(x,z)+d_X(y,z)+1]
$$ 
for any $x,y\in X$. 

\end{example}

Denote the $k$-neighborhood of $A$ by $N_k(A)$, i.\,e. 
$$
N_k(A)=\{x\in X:d_X(x,A)\leq k\}. 
$$
For $B\subset X$, denote by $H_B=l_2(B)\subset l_2(X)=H_X$ the closed subspace in $l_2(X)$ generated by the functions $\delta_x$, $x\in B$. To simplify notation, we will identify an operator $T\in\mathbb B(H_B)$ with the operator in $\mathbb B(H_X)$ equal to $T$ on $H_B$ and equal to 0 on $H_B^\perp$.

\begin{proposition}
The module $M_{X,d^A}$ is canonically isomorphic to the norm closure of the set 
$\bigcup_{k=1}^\infty C_u^*(N_k(A))$.

\end{proposition}
\begin{proof}
Let $T\in C_u^*(N_k(A))$ be an operator of propagation not exceeding $L$. If $T_{x_0,y_1}\neq 0$ then $d_X(x,y)\leq L$ and $x,y\in N_k(A)$, hence there exists a point $u\in A$ such that $d_X(x,u)\leq k+1$. Then, taking $z=u$, we obtain 
\begin{eqnarray*}
d^A(x_0,y_1)&=&\inf_{z\in A}[d_X(x,z)+d_X(z,y)+1]\leq d_X(x,u)+d_X(u,y)+1\leq\\
&\leq& k+1+d_X(u,x)+d_X(x,y)+1\leq L+2k+3,
\end{eqnarray*}
i.\,e. $T$ is of finite propagation, $T\in M_{X,d^A}$.

Let now $S\in M_{X_1,d^A}$ be an operator of propagation not exceeding $L$. If $S_{x_0,y_1}\neq 0$ then 
$$
d^A(x_0,y_1)=\inf_{z\in A}[d_X(x,z)+d_X(z,y)+1]\leq L. 
$$
The triangle inequality implies that
\begin{equation}\label{treug}
d_X(x,y)\leq d_X(x,z)+d_X(z,y) 
\end{equation}
for any $z\in X$, hence, passing in (\ref{treug}) to the infimum with respect to $z\in X$, we get 
$$
d_X(x,y)\leq d^A(x_0,y_1)-1\leq L-1.
$$  

As a metric is always non-negative, we have
$$
d_X(x,A)=\inf_{z\in A}d_X(x,z)\leq\inf_{z\in A}[d_X(x,z)+d_X(z,y)]= d^A(x_0,y_1)-1\leq L-1. 
$$
Similarly, we obtain that $d_X(y,A)\leq L$. Thus, $S_{x_0,y_1}\neq 0$ implies that $x,y\in N_L(A)$ and $d_X(x,y)\leq L-1$, i.\,e. $S\in C_u^*(N_L(A))$.   
\end{proof}

\begin{lemma}
Let $A\subset X$. If $X\setminus N_k(A)$ is not empty for any $k\in\mathbb N$ then the submodule
$M_{X,d^A}$ in the module $C_u^*(X)$ is not orthogonally complemented.

\end{lemma}
\begin{proof}
Let $B\subset X$, $P_B$ a projection onto $l_2(B)$ in $l_2(X)$. Evidently, $P_{N_k(A)}\in M_{X,d^A}$ for any $k\in\mathbb N$. If $S\in C_u^*(X)$ is orthogonal to $M_{X,d^A}$ then $P_{N_k(A)}S=0$ for any $k\in\mathbb N$, but, as $\bigcup_{k=1}^\infty N_k(A)=X$, the sequence $P_{N_k(A)}$ of projections is convergent to $1=P_X$ with respect to the strong topology, whence $S=0$. 

By assumption, $M_{X,d^A}\neq C_u^*(X)$. Indeed, if $1\in C_u^*(X)$ would belong to $M_{X,d^A}$ then there would exist a sequence $T^{(k)}\in C_u^*(N_k(A))$ such that $T^{(k)}$ would converge to the unit with rspect to the norm topology. But if $x\notin N_k(A)$ then $T^{(k)}\delta_x=0$, so the convergence may take place only with respect to the strong topology, but not in norm. 
 \end{proof}

Two extreme examples are the cases when $A=X$ and when $A$ consists of a single point. In the first case $M_{X,d^X}=C_u^*(X)$, and the second case is given by the following statement. 
\begin{proposition}
Let $x_0\in X$. If $X$ is proper, i.\,e. if each ball contains only a finite number of points, then $M_{X,d^{\{x_0\}}}=\mathbb K(H_X)$.

\end{proposition}
\begin{proof}
Properness of $X$ means that the operator $T\in M_{Y,d^{\{x_0\}}}$ is of finite propagation if and only if it is of finite rank. Passing to the closure, we obtain the required statement.
\end{proof}

\section{Case $Y=X\times\mathbb N$}

Consider the case $Y=X\times\mathbb N$. Let us introduce the notation $\overline{\mathbb N}=\mathbb N\cup\{0\}$. For convenience, we write $X_n$ instead of $X\times\{n\}\subset Y$, and $X_0=X$. For the point $x\in X$, we denote the point $(x,n)\in X\times\mathbb N$, $n\in\overline{\mathbb N}$, by $x_n\in X_n$.

In this case $H_{X\times\mathbb N}=\oplus_{n=1}^\infty H_{X_n}$. By $Q_n:H_{X\times\mathbb N}\to H_{X_n}$ we denote the projection onto the $n$-th direct summand.
For $T:H_X\to H_{X\times\mathbb N}$, put $T_n=Q_nT:H_X\to H_{X_n}$. In what follows, we identify $T$ with the sequence $(T_n)_{n\in\mathbb N}$.

Consider first the metric $d_1$ on $X\sqcup X\times\mathbb N=X\times\overline{\mathbb N}$ given by the formula $d_1(x_n,y_m)=d_X(x,y)+|nm|$ for any $x,y\in X$, $n,m\in\overline{\mathbb N}$, i.\,e. coinciding on the factors $X$ and $\overline{\mathbb N}$ with the metric $d_X$ and with the standard metric on $\overline{\mathbb N}$, respectively.

\begin{proposition}
The module $M_{X\times\mathbb N,d_1}$ coincides with the standard Hilbert $C^*$-module $l_2(C_u^*(X))$.  

\end{proposition}
\begin{proof}
Let the operator $T=(T_n)_{n\in\mathbb N}$ have a propagation not exceeding $L$. Then $T_n=0$ for any $n>L$, therefore the sequence $(T_n)$ consists of zeroes for $n>L$. On the other hand, any sequence $(T_1,T_2,\ldots,T_n,0,0,\ldots)$, where $T_i\in C_u^*(X)$, $i=1,\ldots,n$, lies in $M_{X\times\mathbb N,d_1}$. Thus, $M_{X\times\mathbb N,d_1}$ is the completion of the set of finitely supported sequences with elements from $C_u^*(X)$, and therefore coincides with $l_2(C_u^*(X))$.
\end{proof}

As a second example, let us consider the metric $d_0$ on $X\sqcup X\times\mathbb N$ determined by the formulas $d_0(x_n,y_m)=d_X(x,y)+1$ for $m\neq n$ and $d_0(x_n,y_n)=d_X(x,y)$ for any $x,y\in X$.

\begin{theorem}
If the space $X$ is bounded then the module $M_{X\times\mathbb N,d_0}$ coincides with $l_2(\mathbb B(H))'=l_2(C_u^*(X))'$. If the space $X$ is unbounded then there exists an element $T\in M_{X\times\mathbb N,d_0}$ such that $T\notin l_2(C_u^*(X))''$.

\end{theorem}
\begin{proof}
The first statement is obvious. Suppose that $X$ is unbounded, then there is a sequence of pairs of distinct points $(x^k,y^k)_{k\in\mathbb N}$ such that $d_X(x^k,y^k)>k$. Put $(T_n)_{x^k,y^k_n}=(T_n)_{y^k,x^k_n}=1$, and all other matrix elements of the matrix of $T_n$ are equal to zero; $(S_n)_{x^k,x^k_n}=(S_n)_{y^k,y^k_n}=1$, and all other matrix elements of the matrix of $S_n$ are equal to zero. Let $T=(T_1,T_2,\ldots)$, $S=(S_1,S_2,\ldots)$. Note that $T_nS_n=S_n$ for any $n\in\mathbb N$, and that the series $\sum_{n=1}^\infty T_n$ and $\sum_{n=1}^\infty S_n$ are convergent with respect to the strong operator topology in $\mathbb B (H_X)$.

The sequence $(T_n)_{n\in\mathbb N}$ defines an element of the dual module $l_2(C_u^*(X))'$. Indeed, each operator $T_n$ has propagation $d_X(x^n,y^n)$, therefore, lies in $C_u^*(X)$, and the partial sums $\sum_{i=1}^N T_n^* T_n$ are uniformly bounded in $N$. However, this sequence does not lie in $M_{X\times\mathbb N,d_0}$ because the strong limit $\sum_{i=1}^\infty T_n^* T_n$ does not lie in $C_u^*(X)$.

The sequence $(S_n)_{n\in\mathbb N}$ lies in $M_{X\times\mathbb N,d_0}$. Indeed, the propagation of each $S_n$, $n\in\mathbb N$, is equal to one, which means that the propagation of $S$ is equal to one. Suppose that $S$ lies in the second dual module $l_2(C_u^*(X))''$. There is a natural action of the first dual module on the second dual. Let $T(S)\in C_u^*(X)$ be the result of this action of $T$ on $S$. This action extends the standard inner product $\langle T,S\rangle=\sum_{i=1}^\infty T_i^* S_i$ when $T,S\in l_2(C_u^*(X))$, but in general the value of $T(S)$ is not related to the series $\sum_{i=1}^\infty T_i^* S_i$, even if this series converges in some topology. Let $L_N\cong C_u^*(X)^N\subset l_2(C_u^*(X))$ be a free submodule of finite sequences of length $N$. Then $l_2(C_u^*(X))=L_N\oplus L_N^\perp$, and similar decompositions into direct sums hold for the first and second dual modules of the module $l_2(C_u^*(X))$. Let $T=T_N+T'_N$, $S=S_N+S'_N$ be the corresponding decompositions of $T$ and $S$, respectively. Let $K_M\subset H_X$ be the $2M$-dimensional linear subspace generated by functions $\delta_{x_n}$ and $\delta_{y_n}$, $n\leq M$. Let us fix $M\in\mathbb N$. Since the series $\sum_{n=1}^\infty S_n^* S_n$ is convergent with respect to the strong topology, for any $\varepsilon>0 $ one can find $N\in\mathbb N$ such that 
$$
\left|\left( \xi,\sum\nolimits_{n=N+1}^\infty S_n^*S_n\eta\right)\right|<\varepsilon
$$ 
for any unit vectors $\xi,\eta\in K_M$. Let $P_M$ be the projection in $H_X$ onto $K_M$. Evidently, $P_M\in C_u^*(X)$. Then 
$$
\|S'_N P_M\|^2=\sup_{\xi,\eta\in K_M,\|\xi\|=\|\eta\|=1}\left|\left(\xi,\sum\nolimits_{n=N+1}^\infty P_MS_n^*S_nP_M\eta\right)\right|\leq\varepsilon,
$$
hence
$$
|(\xi,S'_N(T)\eta)|=|(\xi,P_MS'_N(T)\eta)|\leq\|S'_NP_M\|\|T\|\leq\varepsilon\|T\|=\varepsilon.
$$
Let $\xi=\delta_{x_n}$, $\eta=\delta_{y_n}$, $M\geq n$. Then 
$$
(\delta_{x_n},S_N(T)\delta_{y_n})=\Bigl(\sum\nolimits_{n=1}^N S^*_nT_n\Bigr)_{x_n,y_n}=T_{x_n,y_n}=1
$$
and
$$
(\delta_{x_n},S(T)\delta_{y_n})=(\delta_{x_n},S_N(T)+S'_N(T)\delta_{y_n}),
$$ 
therefore, $|(\delta_{x_n},S(T)\delta_{y_n})-1|\leq \varepsilon$. Thus, if $\varepsilon<\dfrac{1}{2}$ then $\bigl|\bigl(S(T)\bigr)_{x_n,y_n}\bigr|\geq\dfrac{1}{2}$. As $n$ was arbitrary, we have $S(T)\notin C_u^*(X)$, which gives a contradiction.
\end{proof}

We don't know if it is true that $l_2(C_u^*(X))''\subset M_{X\times\mathbb N,d_0}$.

\section{Case $X=\mathbb N^2$}

One of the simplest unbounded spaces is the space $X=\mathbb N^2=\{k^2: k\in\mathbb N\}$ with the standard metric. In particular, its asymptotic dimension \cite{gromov} is zero. Here we consider some examples of Hilbert $C^*$-modules over the uniform Roe algebra of this space. This algebra has a simple description: it is the sum (but not a direct sum) of two $*$-subalgebras: the subalgebra of compact operators and the subalgebra of diagonal operators, $C_u^*(X)=\mathbb K(H_X)+\mathbb D(H_X)$. Denote by $l_1(\overline{\mathbb N})$ the Banach space of absolutely summable sequences $(t_0,t_1,t_2,\ldots)$, $t_i\in\mathbb R$, $i\in\overline{\mathbb N}$, with the standard $l_1$-metric.

For convenience, we shall denote the point $k^2\in X$ by $x^k$.
Let 
$$
x^k=x_0^k=(k^2,0,\ldots,0,1,0,0,\ldots),\quad X=X_0=\{x_0^k:k\in\mathbb N\}, 
$$
and let $X_n=\{x_n^k:k\in\mathbb N\}$ be the $n$-th copy of the space $X$.
Consider several examples of metrics induced by various embeddings $Y=\bigsqcup_{n=1}^\infty X_n\subset l_1(\overline{\mathbb N})$.

\begin{example}
Let  $x^k_n=(k^2,0,\ldots,0,1,0,0,\ldots)$ for $k\geq n$, where 1 is the $n$-th coordinate. 
For $k<n$, we place the points $x^k_n$ on the ray 
$$
t_0=\cdots=t_{n-1}=t_{n+1}=t_{n+2}=\cdots=0,\quad t_n\geq 0,
$$ 
in such a way that the distance between them coincide with the metric $d_X$, i.\,e. $d(x^k_n,x^l_n)=d_X(x^k,x^l)$, $k,l\in\mathbb N$, where $d$ is the metric on $\bigsqcup_{n=0}^\infty X_n$ induced by the metric on $l_1(\overline{\mathbb N})$. Then
$\lim_{n\to\infty}d(x^k_0,x^k_n)=\infty$, and $d(x^k_0,x^k_n)=1$ for $k\geq n$.  

\end{example}

Denote by $l_2(\mathbb D(H_X))'_0\subset l_2(\mathbb D(H_X))'$ the submodule consisting of sequences $(D_n)_{n\in\mathbb N}$, 
$$
D_n=\operatorname{diag}(d_n^1,d_n^2,\ldots)\in\mathbb D(H_X),\qquad d_n^i\in\mathbb C,\quad i,n\in\mathbb N, 
$$
such that $d_n^i=0$ for $i<n$.

\begin{proposition}\label{Lemma-primer}
$M_{X\times\mathbb N,d}= l_2(\mathbb K(H_X))+l_2(\mathbb D(H_X))'_0$.

\end{proposition}
\begin{proof}
If $T=(T_n)_{n\in\mathbb N}\in l_2(\mathbb K(H_X))$ then for any $\varepsilon>0$ there exist $K_1,\ldots,K_n\in\mathbb K(H_X)$ such that $\|T-K\|<\varepsilon$, where $K=(K_1,\ldots,K_n,0,0,\ldots)$. As the propagation of the operator $K$ is finite, we have $K\in M_{Y,d}$. If $T=(T_n)_{n\in\mathbb N}\in l_2(\mathbb D(H_X))'_0$ then the propagation of $T$ equals one, hence, $T\in M_{X\times\mathbb N,d}$.  

Let the propagation of $T\in M_{Y,d}$ does not exceed $L$. It follows from $T_{x^k_0,x_n^l}\neq 0$ that $d(x^k_0,x^l_n)\leq L$, hence, if $n\geq L$ then $k=l\geq n$. Denote by $P_L$ the projection onto the linear span of the functions $\delta_{x_1},\ldots,\delta_{x_L}$. Let $D_n$ be the diagonal part of $T_n$, i.\,e. the operator given by $D_n\delta_x=(T_n)_{x_0,x_n}\delta_{x_n}$, $x\in X$, and let $K_n=P_L(T_n-D_n)P_L$. Then the rank of $K_n$ does not exceed $L$ for $n\in\mathbb N$, and $K_n=0$ for $n\geq L$. Set 
$$
K=(K_1,K_2,\ldots),\quad D=(D_1,D_2,\ldots),\quad T=K+D.
$$ 

Let now $\{T^{(L)}\}_{L\in\mathbb N}$ be norm convergent to $T$, and let the propagation of $T^{(L)}$ does not exceed $L$. Then the diagonal parts $D^{(L)}$ of the operators $T^{(L)}$ are norm convergent to the diagonal part $D$ of the operator $T$, and, therefore, the same holds for their compact parts: $K^{(L)}\to K$ as $L\to\infty$. Since $K_n^{(L)}=0$ for $n> L$, the norm closure of the finite support sequence lies in $l_2(\mathbb K(H_X))$. 

Let us show that the partial sums $\|\sum_{n=1}^N D^*_nD_n\|$ are uniformly bounded in $N$. If this would be false then for any $m>0$ there would exist $N_m$ such that $\|\sum_{n=1}^{N_m} D^*_nD_n\|>m$. As $D^{(L)}$ is norm convergent to $D$, there exists $L_0>0$ such that $\|D^{(L)}\|\leq \|D\|+1$ for any $L\geq L_0$. But 
$$
\bigl\|\sum\nolimits_{n=1}^{N_m}(D^{(L)}_n)^*D^{(L)}_n\bigr\|\leq \|D^{(L)}\|^2\leq (\|D\|+1)^2. 
$$
Taking here $m>(\|D\|+1)^2$, we obtain a contradiction. Thus, $D\in l_2(\mathbb D(H_X))'$. It is easy to see that $D$ lies in the submodule $l_2(\mathbb D(H_X))'_0$.
\end{proof}

\begin{example}
Let 
$$
x^k_n=\left\lbrace\begin{array}{ll}(k^2,0,\ldots,0,1,0,0,\ldots)& \mbox{for\ } k<n;\\ (k^2-n,0,\ldots,0,n+1,0,0,\ldots)& \mbox{for\ }k\geq n, \end{array}\right.
$$
where the non-zero entries are at the 0-th and $n$-th place.
Let $\rho$ be the metric on $\bigsqcup_{n=0}^\infty X_n$ induced by the metric on $l_1(\overline{\mathbb N})$. Then
$\lim_{n\to\infty}d(x^k_0,x^k_n)=\infty$, and $d(x^k_0,x^k_n)=1$ for $k\geq n$.  

\end{example}

Let $l_2(\mathbb D(H_X))'_1\subset l_2(\mathbb D(H_X))'$ be a submodule consisting of sequences $(D_n)_{n\in\mathbb N}$, 
$$
D_n=\operatorname{diag}(d_n^1,d_n^2,\ldots)\in\mathbb D(H_X),\qquad d_n^i\in\mathbb C,\quad i,n\in\mathbb N, 
$$
such that $d_n^i=0$ for $i> n$.

\begin{proposition}
$M_{X\times\mathbb N,\rho}= l_2(\mathbb K(H_X))+l_2(\mathbb D(H_X))'_1$.

\end{proposition}
\begin{proof}[Proof\nopunct] is similar to that of Proposition \ref{Lemma-primer}.
\end{proof}

Let a mapping $\varphi:\mathbb N\to\mathbb N$ satisfy the following conditions:
\begin{enumerate}
\item[1)]
$\varphi$ takes each value infinitely many times,
\item[2)]
$\varphi(k)\leq k$ for any $k\in\mathbb N$.
\end{enumerate}

\begin{example}
Set 
$$
x_n^k=(k^2-\varphi(k),0,\ldots,0,\varphi(k),0,0,\ldots), 
$$
where $\varphi(k)$ is the $(n+1)$-th coordinate. The metric on $X_0=\{x_0^k:k\in\mathbb N\}$ coincides with the standard metric on $X=\mathbb N^2$. Let $b$ be the metric on $\bigsqcup_{n=0}^\infty X_n$ induced by the metric on $l_1(\overline{\mathbb N})$. Let $\{k_i\}_{i\in\mathbb N}$ be a sequence such that $\varphi(k_i)=1$ for any $i\in\mathbb N$. Then $b(x_0^{k_i},x_n^{k_i})=2$ for any $i\in\mathbb N$. Put 
$$
(T_n)_{x_0^k,x_n^l}=\left\lbrace\begin{array}{cl}1,&\mbox{if\ }k=l=k_n;\\0&\mbox{otherwise.}\end{array}\right.
$$  
Then $T=(T_n)_{n\in\mathbb N}\in M_{X\times\mathbb N,b}$.

\end{example}

\end{document}